\documentclass{amsart}[12pt]
\usepackage{amssymb,amsmath}
\usepackage{enumerate}
\usepackage[english]{babel}
\usepackage{color}

\newcommand {\ep} {\epsilon}

\newcommand {\ii} {\infty}
\newcommand {\dt} {\delta}
\newcommand {\al} {\alpha}
\newcommand {\bt} {\beta}
\newcommand {\lb} {\lambda}

\newcommand {\su} {\subset}

\newcommand {\mc} {\mathcal}

\newtheorem{teo}{Theorem}[section]
\newtheorem{pro}{Proposition}[section]

\theoremstyle{definition}
\newtheorem{rem}{Remark}[section]
\newtheorem{df}{Definition}[section]

\title{Individual ergodic theorems in noncommutative symmetric spaces}
\begin{document}
\date{April 1, 2016}

\begin{abstract}
It is known that, for a positive Dunford-Schwartz operator in a noncommutative $L^p-$space, $1\leq p<\ii$ or,
more generally, in a noncommutative Orlicz space with order continuous norm, the
corresponding ergodic averages converge bilaterally almost uniformly.
We show that these averages converge bilaterally almost uniformly in each noncommutative
symmetric space $E$ such that $\mu_t(x) \to 0$  as $t \to 0$ for every $x \in E$,
where $\mu_t(x)$ is a non-increasing rearrangement of $x$.
In particular, these averages converge bilaterally almost uniformly in all noncommutative symmetric
spaces with order continuous norm.
\end{abstract}

\author{Vladimir Chilin}
\address{The National University of Uzbekistan}
\email{vladimirchil@gmail.com}
\author{Semyon Litvinov}
\address{Pennsylvania State University Hazleton}
\email{snl2@psu.edu}

\maketitle

\section{Preliminaries}
Let $\mc M$ be a semifinite von Neumann algebra equipped with a faithful normal semifinite trace $\tau$.
Let $\mc P(\mc M)$ stand for the set of projections in $\mc M$. If $\mathbf 1$ is the identity of $\mc M$ and
$e\in \mc P(\mc M)$, we write $e^{\perp}=\mathbf 1-e$. Denote by $L^0=L^0(\mc M)=L^0(\mc M,\tau)$ the $*$-algebra of
$\tau$-measurable operators affiliated with $\mc M$. Let $\| \cdot \|_{\ii}$ be the uniform norm in $\mc M$.
Equipped with the {\it measure topology} given by the system
$$
V(\ep,\dt)=\{ x\in L^0: \ \| xe\|_{\ii}\leq \dt \text{ \ for some \ } e\in \mc P(\mc M) \text{ \ with \ } \tau(e^{\perp})\leq \ep \},
$$
$\ep>0$, $\dt>0$, $L^0$ is a complete metrizable topological $*$-algebra \cite{ne}.

Let $x\in L^0$, and let $\{ e_{\lb}\}_{\lb\ge 0}$ be the spectral family of projections for the absolute value $| x|$ of $x$.
If $t>0$, then the {\it $t$-th generalized singular number of $x$ (a non-increasing rearrangement of $x$)}
(see \cite{fk}) is defined as
$$
\mu_t(x)=\inf\{\lb>0: \tau(e_{\lb}^{\perp})\leq t\}.
$$

A Banach space $(E, \| \cdot \|_E)\su L^0$ is called {\it fully symmetric} if the conditions
$$
x\in E, \ y\in L^0, \ \int \limits_0^s\mu_t(y)dt\leq  \int \limits_0^s\mu_t(x)dt \text{ \ for all\ }s>0 \ \  (\text {writing } \  y \prec\prec x)
$$
imply that $y\in E$ and $\| y\|_E\leq \| x\|_E$. It is known \cite{ddp} that if $(E, \| \cdot \|_E)$ is a fully symmetric space,
$x_n,x\in E$, and $\|x- x_n\|_E\to 0$, then $x_n\to x$ in measure. A fully symmetric space $(E, \| \cdot \|_E)$ is said to possess
{\it Fatou property} if conditions
$$
x_{\al}\in E^+, \ \ x_{\al}\leq x_{\bt} \text{ \ for } \al \leq \bt, \text{\ and\ }\sup_{\al} \| x_{\al}\|_E<\ii
$$
imply that there exists $x=\sup \limits_{\al}x_{\al}\in E$ and $\| x\|_E=\sup \limits_{\al} \| x_{\al}\|_E$.
A space $(E, \| \cdot \|_E)$ is said to have {\it order continuous norm} if $\| x_{\al}\|_E\downarrow 0$ whenever $x_{\al}\in E$ and $x_{\al}\downarrow 0$.

Let $L^0(0,\ii)$ be the linear space of all (equivalence classes of) almost everywhere finite complex-valued Lebesgue measurable
functions on the interval $(0,\ii)$. We identify $L^{\ii}(0,\ii)$ with the commutative von Neumann algebra acting
on the Hilbert space $L^2(0,\ii)$ via
multiplication by the elements from $L^{\ii}(0,\ii)$ with the trace given by the integration
with respect to Lebesgue measure. A Banach space $E\su L^0(0,\ii)$ is called {\it fully symmetric function space} on $(0,\ii)$
if the condition above holds with respect to the von Neumann algebra $L^{\ii}(0,\ii)$.

Let $E=(E(0,\ii), \| \cdot \|_E)$ be a fully symmetric function space.
For each $s>0$ let $D_s: E \to E$ be the bounded linear operator given by $D_s(f)(t) = f(t/s), \ t > 0$.
The {\it Boyd indices} $p_E$ and $q_E$ are defined as
$$
p_E=\lim\limits_{s\to\infty}\frac{\log s}{\log \|D_s\|_E}, \ \ q_E=\lim\limits_{s \to +0}\frac{\log s}{\log \|D_{s}\|_E}.
$$
It is known that $1\leq p_E\leq q_E\leq \ii$ \cite[II, Ch.2, Proposition 2.b.2]{lt}.
A fully symmetric function space is said to have {\it non-trivial Boyd indices} if $1<p_E$ and $q_E<\ii$.
For example, the space $L^p(0,\ii)$, $1< p<\ii$, have non-trivial Boyd indices:
$$
p_{L^p(0,\infty)} = q_{L^p(0,\infty)} = p
$$
\cite[Ch.4, \S 4, Theorem 4.3]{bs}.

A Banach lattice  $(E,\|\cdot\|_E)$  is called {\it $q$-concave}, $1 \leq q < \infty$, if there exists a 
constant  $M>0$ such that 
$$
\left (\sum_{i=1}^n\|x_i\|^q\right )^{\frac{1}{q}} \leq M \left \| \left (\sum_{i=1}^n|x_i|^q\right )^{\frac{1}{q}}\right \|_E
$$
for every finite set  $\{x_i\}_{i=1}^n \su E$.
Let  $E=(E(0,\ii), \| \cdot \|_E)$ be a fully symmetric function space, and let $E$ be $q$-concave for same $1 \leq q < \infty$. 
Then  there does not exist a sublattice of $E$ isomorphic to $l^{\infty}$, hence the norm $\| \cdot \|_E$ is order continuous \cite[Corollary 2.4.3]{pm}.

\vskip 5 pt
If $E(0,\ii)$ is a fully symmetric function space, define
$$
E(\mc M)=E(\mc M, \tau)=\{ x\in L^0: \ \mu_t(x)\in E\}
$$
and set
$$
\| x\|_{E(\mc M)}=\| \mu_t(x)\|_E,  \ x\in E(\mc M).
$$
It is shown in \cite{ks} that $(E(\mc M), \| \cdot \|_{E(\mc M)})$ is a fully symmetric space.
If $1\leq p<\ii$ and $E=L_p(0,\ii)$, the space $(E(\mc M), \| \cdot \|_{E(\mc M)})$ coincides with the noncommutative
$L^p-$space $L^p=L^p(\mc M)=(L^p(\mc M, \tau), \| \cdot \|_p)$ because
$$
\| x\|_p=\left (\int \limits_0^{\ii}\mu_t^p(x)dt\right )^{1/p}=\| x\|_{L^p(\mc M)}
$$
(see \cite[Proposition 2.4]{ye0}). In addition, $L^{\infty}(\mc M)=\mc M$ and
$$
 (L^1\cap L^{\infty})(\mc M) = L^1(\mc M)\cap \mc M,  \ \|x\|_{L^1\cap L^{\infty}}=\max \{ \|x\|_1, \|x\|_{\mc M}\},
$$
$$
(L^1 + L^\ii)(\mc M) = L^1(\mc M) + \mc M,
$$
$$
\|x\|_{L^1+L^{\infty}}=\inf\{ \|x\|_1+ \|y\|_{\mc M}: \ x \in L^1(\mc M), \ y \in \mc M\}= \int_0^1 \mu_t(x) dt
$$
(see \cite[Proposition 2.5]{ddp}). (For a comprehensive review of noncommutative
$L^p-$spaces, see \cite{ye0, px}.)

Since for a fully symmetric function space $E(0,\ii)$,
$$
L^1(0,\ii) \cap L^{\infty}(0,\ii) \su E \su L^1(0,\ii) + L^{\infty}(0,\ii)
$$
with continuous embeddings \cite[Ch.II, \S 4, Theorem 4.1]{kps}, we also have
$$
L^1(\mc M) \cap \mc M \subset E(\mc M, \tau) \subset L^1(\mc M) + \mc M,
$$
with continuous embeddings.

Denote
$$
L^0_{\tau} = \{x \in L^0(\mc M):  \ \mu_t(x) \rightarrow 0 \text{ \ as \ } t\rightarrow 0\}.
$$
Observe that $L^0_{\tau}$ is a linear subspace of $L^0$ which is solid in the sense that if
$x \in L^0_{\tau}$ and if $y \in  L^0$ and $\mu_t(y) \leq \mu_t(x)$, then also $y \in  L^0_{\tau}$. It is clear that $x \in L^0_{\tau}$  if and only if $\tau(\chi_{(\lambda,\infty)}(|x|)) < \infty$ for all $\lambda > 0$; moreover, $L^1 \cap \mc M \subset L^0_{\tau}$.

Define
$$
\mc R_\tau= \{x \in L^1 + \mc M: \ \mu_t(x) \to 0 \text{ \ as \ } t\to  0 \}.
$$
It follows from the next proposition that $(\mc R_\tau, \|\cdot\|_{L^1+\mc M})$
is a Banach space.

\begin{pro}\cite[Proposition 2.7]{ddp}\label{p11}
$\mc R_\tau$ is the closure of $L^1 \cap \mc M$  in $L^1 + \mc M$.
\end{pro}

Note that definitions of $L^0_{\tau}$ and $\| \cdot \|_{L^1+L^\ii}$ yield that if
$$
x\in \mc R_\tau, \ y\in L^1+L^\ii, \text{\ and \ } \mu_t(y) \leq \mu_t(x),
$$
then $y\in \mc R_\tau$ and $\| y\|_{L^1+\mc M} \leq \| x\|_{L^1+\mc M}$. Therefore 
$(\mc R_\tau, \|\cdot\|_{L^1+\mc M})$ is a noncommutative symmetric space \cite[\S 2]{ddp}.

\section{Individual Ergodic Theorems in $\mc R_\tau$}

A linear operator $T: L^1+\mc M\to  L^1+\mc M$ is called a {\it Dunford-Schwartz operator} (writing $T\in DS)$ if
$$
\| T(x)\|_1\leq \| x\|_1 \ \ \forall \ x\in L^1 \text{ \ and \ } \| T(x)\|_{\ii}\leq \| x\|_\ii \ \ \forall \ x\in \mc M.
$$
If a Dunford-Schwartz operator $T$ is positive, that is, $T(x)\ge 0$ whenever $x\ge 0$, we shall write $T\in DS^+$.

If $T\in DS$, then
\begin{equation}\label{e11}
\| T\|_{ L^1+\mc M\to  L^1+\mc M} \leq 1
\end{equation}
(see, for example, \cite[\S 4]{ddp}) and, in addition, $Tx \prec\prec x$ for all $x \in L^1+\mc M$ \cite[Theorem 4.7]{ddp}. Hence
\begin{equation}\label{e12}
T(\mc R_\tau) \subset \mc R_\tau.
\end{equation}
Given $T\in DS$ and $x\in L^1+L^{\ii}$, denote
\begin{equation}\label{e13}
A_n(x)=A_n(T,x)=\frac 1n\sum_{k=0}^{n-1} T^k(x), \ \ n=1,2,\dots,
\end{equation}
the corresponding ergodic averages of the operator $x$.

Note that, by (\ref{e11}) and (\ref{e12}), for any $T\in DS$
its restriction on $\mc R_\tau$ is a linear contraction (also denoted by $T$).

\begin{df}
A sequence  $\{ x_n\}\su L^0$ is said to converge to $x\in L^0$ {\it bilaterally almost
uniformly (b.a.u.)} if for every $\ep>0$ there exists such a projection $e\in \mc P(\mc M)$ that $\tau(e^{\perp})\leq \ep$ and
$\| e(x-x_n)e\|_{\ii}\to 0$.
\end{df}

\vskip 5pt
The following groundbreaking result was established in \cite{ye} as a corollary of a
noncommutative maximal ergodic inequality \cite[Theorem 1]{ye}
(for the assumption $T\in DS^+$, see \cite[Remark 1.2]{cl} and \cite[Lemma 1.1]{jx}).

\begin{teo}\label{t12}
Let $T\in DS^+$ and $x\in L^1$. Then the averages (\ref{e13}) converge b.a.u. to some $\widehat x\in L^1$.
\end{teo}

 The next result is an extension of Theorem \ref{t12} to $\mc R_\tau$.
\begin{teo}\label{t13}
Let $T\in DS^+$ and $x\in \mc R_\tau$. Then the averages (\ref{e13}) converge b.a.u. to some $\widehat x\in  L^1+\mc M$.
\end{teo}
\begin{proof}
Without loss of generality assume that $x\geq 0$.  Let $\{ e_\lb\}_{\lb\ge 0}$ be
the spectral family of $x$. Given $k \in \mathbb N$, denote $x_k = \int_{1/k}^\ii \lb de_\lb$
and $y_k = \int_0^{1/k} \lb de_\lb$. Then $0 \leq y_k \leq {\frac1k} \cdot \mathbf 1$, $x_k \in L^1$,
and $x = x_k + y_k$ for all $k$.

Fix $\ep > 0$. By Theorem \ref{t12}, $A_n(x_k) \to \widehat x_k \in L^1$ b.a.u. for each $k$.
Therefore there exists $e_k\in \mc P(\mc M)$ such that $\tau(e_k^{\perp})\leq \frac \ep{2^k}$ and
$\| e_k(A_n(x_k)-\widehat x_k)e_k\|_\ii\to 0$ as $n \rightarrow \infty$.

Then it follows that
$$
\| e_k(A_n(x_k)-A_m(x_k))e_k\|_{\ii} < \frac1k \text{ \ \ for all \ }  m, n \geq N(k).
$$
Since $\|y_k\|_{\ii} \leq \frac1k$, we have
$$
\| e_k(A_n(x)-A_m(x))e_k\|_{\ii} \leq
$$
$$
\leq \| e_k(A_n(x_k)-A_m(x_k))e_k\|_{\ii} + \| e_k(A_n(y_k)-A_m(y_k))e_k\|_{\ii} <
$$
$$
 <\frac1k + \| e_kA_n(y_k)e_k\|_{\ii} + \| e_kA_m(y_k)e_k\|_{\ii}  \leq {\frac3k}.
$$
for each $k$ and all $m, n \geq N(k)$.

Let $e = \bigwedge \limits_{k\geq 1} e_k$. Then $\tau(e^\perp) \leq \sum_{k=1}^{\infty} \tau(e_k^\perp) \leq \ep$ and
$$
\| e(A_n(x)-A_m(x))e\|_{\ii}  < {\frac3k}
$$
for all $m, n \geq N(k)$. This means that $\{A_n(x)\}$ is Cauchy with respect to b.a.u. convergence. Then, by \cite[Theorem 2.3] {cls}, we conclude that the sequence $\{A_n(x)\}$ converges b.a.u. to some $\widehat x\in L^0$.

Since  $L^1+\mc M$ satisfies Fatou property \cite[\S 4]{dp}, its unit ball is closed in the measure topology \cite[Theorem 4.1]{ddst},  and (\ref{e11}) implies that $\widehat x \in L^1+\mc M$.
\end{proof}

Let $\Bbb C_1=\{z\in \Bbb C: |z|=1\}$ be the unit circle in $\mathbb C$. A function $P : \mathbb Z \to \mathbb C$ is said to be a {\it trigonometric polynomial} if
$P(k)=\sum_{j=1}^{s} z_j\lb_j^k$, $k\in \mathbb Z$, for some $s\in \mathbb N$, $\{ z_j \}_1^s \subset \mathbb C$, and $\{ \lb_j \}_1^s \subset \mathbb C_1$.
A sequence $\{ \beta_k \}_{k=0}^{\ii} \subset \Bbb C$ is called a {\it bounded Besicovitch sequence} if

(i) $| \beta_k | \leq C < \ii$ for all $k$;

(ii) for every $\ep >0$ there exists a trigonometric polynomial $P$ such that
$$
\limsup_n \frac 1{n+1} \sum_{k=0}^n | \beta_k - P(k) | < \ep.
$$

The following theorem was established in \cite{cls}.

\begin{teo}\label{t14}
Assume that $\mc M$ has a separable predual. Let $T\in DS^+$ and let $\{ \bt_k\}$ be a Besicovitch sequence.
Then for every $x\in L^1(\mc M)$ the averages
\begin{equation}\label{e14}
\frac 1n\sum_{k=0}^{n-1}\bt_kT^k(x)
\end{equation}
converge b.a.u. to some $\widehat x\in L^1$.
\end{teo}

In view of Theorem \ref{t14} and since a sequence $\{ \bt_k\}$ in question is bounded, the proof of Theorem \ref{t13} can be
carried out for the averages (\ref{e14}), and we obtain the following.

\begin{teo}\label{t15}
Let $\mc M$, $T$, $\{ \bt_k\}$ be as in Theorem \ref{t14}. Then for every $x\in \mc R_\tau$ the averages (\ref{e14})
converge to some $\widehat x\in L^1+\mc M$.
\end{teo}

\section{Applications of Theorem \ref{t13}}
Let us present some examples of noncommutative symmetric spaces for which 
Dunford-Schwartz individual ergodic theorem is valid.

1. Let $\Phi$ be an Orlicz function, $L^\Phi=L^\Phi(\mc M,\tau)$ the corresponding noncommutative
Orlicz space, $\| \cdot\|_\Phi$ the Luxemburg norm in $L^\Phi$ (see \cite{cl2}). Since $(L^\Phi(0,\ii), \| \cdot\|_\Phi)$ has Fatou property, the space $(L^\Phi(\mc M,\tau), \| \cdot\|_\Phi)$ has it as well \cite[\S 5]{ddp}.
As shown in the proof of \cite[Proposition 2.1]{cl2}, $L^\Phi \subset \mc R_\tau$.
It follows then from Theorem \ref{t13} that we have

\begin{teo}\label{t21}
Let $T\in DS^+$ and $x\in L^\Phi$. Then the averages (\ref{e13}) converge b.a.u. to some $\widehat x\in L^\Phi$.
\end{teo}

2. Let $(E(\mc M), \|\cdot\|_{E(\mc M)})$  be a noncommutative fully symmetric space  with order continuous norm. Then $\tau(\{|x| > \lambda\}) < \infty$  for all $x \in E(\mc M)$ and $\lambda > 0$, so  $E(\mc M) \subset \mc R_\tau$. Thus we have
\begin{teo}\label{t22}
Let $(E(\mc M), \|\cdot\|_{E(\mc M)})$  be noncommutative fully symmetric space  with order continuous norm. Let $T\in DS^+$ and $x\in E(\mc M)$. Then the averages (\ref{e13}) converge b.a.u. to some $\widehat x\in E(\mc M)$.
\end{teo}

3. Let $\varphi$ be an increasing continuous concave function on $[0, \infty)$ with $\varphi(0) = 0$ and
$\varphi(\infty) = \ii$, and let $\Lambda_\varphi(\mc M)$ be the corresponding
noncommutative Lorentz space (see, for example, \cite{cs}). Since $\Lambda_\varphi(0,\ii)$
is a fully symmetric space  with order continuous norm and Fatou property \cite[Ch. II, \S 5]{kps},
$\Lambda_\phi(\mc M)$ also has order continuous norm and Fatou property \cite[Proposition 3.6 and \S 5]{ddp}. Then  we have
\begin{teo}\label{t23}
Let $T\in DS^+$ and $x\in \Lambda_\varphi(\mc M)$. Then the averages (\ref{e13}) converge b.a.u. to some $\widehat x\in \Lambda_\varphi(\mc M)$.
\end{teo}

4.  Let $E=(E(0,\ii), \|\cdot\|_{E(0,\ii)})$  be a fully symmetric function space, and let 
$E$ be $q$-concave for some $1 \leq q < \infty$.  Then $E$,
hence $(E(\mc M), \|\cdot\|_{E(\mc M)})$, has order continuous norm.  
Therefore we have the following.

\begin{teo}\label{t24}
Let $E=(E(0,\ii), \|\cdot\|_{E(0,\ii)})$ be a fully symmetric function space. Assume that $E$ is $q$-concave 
for same $1 \leq q < \infty$. Let $T\in DS^+$ and $x\in E(\mc M)$. 
Then the averages (\ref{e13}) converge b.a.u. to some  $\widehat x\in L^1+\mc M$. In addition,  if 
$E$ has  Fatou property, then $\widehat x\in E(\mc M)$.
\end{teo}

5. Let $E=(E(0,\ii), \|\cdot\|_{E(0,\ii)})$  be a fully symmetric function space.
Assume that $E$ has  non-trivial Boyd indices. According to  \cite[II, Ch.2, Proposition 2.b.3]{lt},
there exist such $1<p, q<\ii$ that the space $E$ is intermediate for the
Banach couple $(L^p(0,\ii), L^q(0,\ii))$. Since
$$
(L^p+L^q)(\mc M)=L^p(\mc M)+L^q(\mc M)
$$
(see \cite[Proposition 3.1]{ddp}), we have
$$
 E(\mc M) \su  L^p(\mc M) +  L^q(\mc M) \su \mc R_\tau.
$$
Then  we also have (cf. \cite[Theorem 4.2]{cl})

\begin{teo}\label{t25}
Let $E=(E(0,\ii), \|\cdot\|_{E(0,\ii)})$ be a fully symmetric function space, and let $E$ have non-trivial Boyd indices. 
Given $T\in DS^+$ and $x\in E(\mc M)$, the averages (\ref{e13}) converge b.a.u. to some  $\widehat x\in L^1+\mc M$. In addition,  if $E$ has  Fatou property, then $\widehat x\in E(\mc M)$.
\end{teo}

\begin{rem}
It is clear that if one assumes that $\mc M$ has a separable predual, 
Theorems \ref{t21} - \ref{t25} remain valid for the Besicovitch weighted averages (\ref{e14}).
\end{rem}


\begin{thebibliography}{99}

\bibitem{bs} C. Bennett, R. Sharpley, {\it Interpolation of Operators}, Academic Press Inc. (London) LTD, 1988.

\bibitem{cs} V. Chilin, F. Sukochev, Weak convergence in non-commutsative symmetrc spaces, {\it J. Operator Theory}, {\bf 31}
(1994), 35-55.

\bibitem{cls} V. Chilin, S. Litvinov, and A. Skalski, A few remarks in non-commutative ergodic theory, {\it J. Operator Theory},
{\bf 53}(2) (2005), 331-350.

\bibitem{cl} V. Chilin, S. Litvinov, Ergodic theorems in fully symmetric spaces of $\tau$-measurable operators, {\it Studia Math.},
{\bf 288}(2) (2015), 177-195.

\bibitem{cl2} V. Chilin, S. Litvinov,  Individual ergodic theorems in noncommutative Orlicz spaces, {\it Positivity}, DOI 10.1007/s11117-016-0402-8.

\bibitem{ddp} P. G. Dodds, T. K. Dodds, and B. Pagter, Noncommutative K\"{o}the duality,
{\it Trans. Amer. Math. Soc.}, {\bf 339}(2) (1993), 717-750.

\bibitem {ddst} P. G. Dodds, T. K. Dodds, F. A. Sukochev, and O. Ye. Tikhonov, A Non-commutative Yoshida-Hewitt theorem and convex sets of measurable operators closed locally in measure, {\it Positivity}, {\bf 9} (2005), 457-484.

\bibitem{dp} P. G. Dodds, B. Pagter, The non-commutative Yosida-Hewitt decomposition revisited, {\it Trans. Amer. Math. Soc.},
{\bf 364}(12)  (2012), 6425-6457.

\bibitem{fk} T. Fack, H. Kosaki, Generalized $s$-numbers of $\tau$-measurable operators, {\it Pacific. J. Math.}, {\bf 123} (1986), 269-300.

\bibitem{jx} M. Junge, Q. Xu, Noncommutative maximal ergodic theorems, {\it J. Amer. Math. Soc.}, {\bf 20}(2) (2007), 385-439.

\bibitem{kps} S. G. Krein, Ju. I. Petunin, and E. M. Semenov, {\it Interpolation of Linear Operators},
Translations of Mathematical Monographs, Amer. Math. Soc., {\bf 54}, 1982.

\bibitem {ks}  N. J. Kalton, F. A. Sukochev, Symmetric norms and spaces of operators, {\it J. Reine Angew. Math.}, {\bf 621} (2008), 81-121.

\bibitem {lt} J. Lindenstraus, L. Tsafriri, {\it Classical Banach spaces I-II}, Springer-Verlag, Berlin Heidelberg New York, 1977.

\bibitem{ne} E. Nelson, Notes on non-commutative integration, {J. Funct. Anal.}, {\bf 15} (1974), 103-116.

\bibitem{px} G. Pisier, Q. Xu, Noncommutative $L^p$-spaces, in {\it Handbook of the geometry of Banach spaces},
{\bf 2} (2003), 1459-1517.

\bibitem{pm} Peter Meyer-Nieber, {\it Banach Lattices}, Springer-Verlag, New York Berlin Heidelberg, 1991.

\bibitem{ye0}  F. J. Yeadon, Non-commutative $L^p$-spaces, {\it Math. Proc. Camb. Phil. Soc.},  {\bf 77} (1975), 91-102.

\bibitem{ye} F. J. Yeadon, Ergodic theorems for semifinite von Neumann algebras. I, {\it J. London Math. Soc.}, {\bf 16}(2) (1977),
326-332.

\end{thebibliography}
\end{document}